\documentclass[12pt]{amsart}
\usepackage{amssymb,amsmath,amsfonts,latexsym}
\usepackage{bm}

\newcommand{\newsection}[1]{\setcounter{equation}{0} \section{#1}}
\newcommand{\bea}{\begin{eqnarray}}
\newcommand{\eea}{\end{eqnarray}}

\newcommand{\vp}{\varphi}

\newcommand{\clb}{\mathcal{B}}

\newcommand{\cle}{\mathcal{E}}
\newcommand{\clf}{\mathcal{F}}

\newcommand{\clh}{\mathcal{H}}
\newcommand{\clk}{\mathcal{K}}
\newcommand{\cll}{\mathcal{L}}
\newcommand{\clm}{\mathcal{M}}
\newcommand{\cln}{\mathcal{N}}

\newcommand{\clq}{\mathcal{Q}}

\newcommand{\cls}{\mathcal{S}}
\newcommand{\clt}{\mathcal{T}}

\newcommand{\clw}{\mathcal{W}}

\newcommand{\tv}{\tilde{V}}
\newcommand{\tw}{\tilde{\mathcal{W}}}
\newcommand{\D}{\mathbb{D}}

\newcommand{\raro}{\rightarrow}

\def\textmatrix#1&#2\\#3&#4\\{\bigl({#1 \atop #3}\ {#2 \atop #4}\bigr)}
\def\dispmatrix#1&#2\\#3&#4\\{\left({#1 \atop #3}\ {#2 \atop #4}\right)}
\newcommand{\be}{\begin{equation}}
\newcommand{\ee}{\end{equation}}
\newcommand{\ben}{\begin{eqnarray*}}
\newcommand{\een}{\end{eqnarray*}}

\newcommand{\NI}{\noindent}

\newcommand{\bi}{\begin{itemize}}
\newcommand{\ei}{\end{itemize}}

\theoremstyle{definition}

\theoremstyle{plain}

\newtheorem{thm}{Theorem}[section]
\newtheorem{cor}[thm]{Corollary}
\newtheorem{lem}[thm]{Lemma}

\theoremstyle{definition}

\numberwithin{equation}{section}

\let\phi=\varphi

\begin{document}

\title[Joint invariant subspaces and $C^*$-algebras]{On certain commuting isometries, joint invariant subspaces and $C^*$-algebras}

\dedicatory{Dedicated to the memory of Professor Ronald G. Douglas, our teacher, mentor and friend}

\author[Das] {B. Krishna Das}
\address{Department of Mathematics, Indian Institute of Technology Bombay, Powai, Mumbai, 400076, India}
\email{dasb@math.iitb.ac.in, bata436@gmail.com}

\author[Debnath]{Ramlal Debnath}
\address{Indian Statistical Institute, Statistics and Mathematics Unit, 8th Mile, Mysore Road, Bangalore, 560059, India}
\email{ramlal\_rs@isibang.ac.in}

\author[Sarkar]{Jaydeb Sarkar}
\address{Indian Statistical Institute, Statistics and Mathematics Unit, 8th Mile, Mysore Road, Bangalore, 560059, India}
\email{jay@isibang.ac.in, jaydeb@gmail.com}




\begin{abstract}
In this paper, motivated by the Berger, Coburn and Lebow and Bercovici, Douglas and Foias theory for tuples of commuting isometries, we study analytic representations and joint invariant
subspaces of a class of commuting $n$-isometries and prove that the $C^*$-algebra generated by the $n$-shift restricted to an invariant subspace of finite codimension in $H^2(\mathbb{D}^n)$ is unitarily equivalent to the $C^*$-algebra generated by the $n$-shift on $H^2(\mathbb{D}^n)$.
\end{abstract}

\maketitle

\newsection{Introduction}

Tuples of commuting isometries on Hilbert spaces are cental objects of study in (multivariable) operator theory. This paper is concerned with the study of analytic representations, joint invariant subspaces and $C^*$-algebras of a certain class of tuples of commuting isometries.

To be precise, let $\clh$ be a Hilbert space, and let $(V_1, \ldots, V_n)$ be an $n$-tuple of commuting isometries on $\clh$. In what follows, we always assume that $n \geq 2$. Set
\[
V = \mathop{\Pi}_{i=1}^n V_i.
\]
We say that $(V_1, \ldots, V_n)$ is an \textit{$n$-isometry} if
$V$ is a shift. A closed subspace
$\cls \subseteq \clh$ is said to be \textit{joint invariant} for
$(V_1, \ldots, V_n)$ if $V_i \cls \subseteq \cls$, $i = 1, \ldots,
n$. Recall that an isometry $X$ on $\clh$ is said to be a \textit{shift}
if $X^{*m} \raro 0$ as $m \raro \infty$ in the strong operator
topology or, equivalently, if $X$ on $\clh$ has no unitary summand.
Moreover, if $X$ is a shift, then $X$ on $\clh$ and $M_z$ on
$H^2_{\clw(X)}(\D)$ are unitarily equivalent, where $\clw(X) = \ker
X^*$, $H^2_{\clw(X)}(\D)$ is the $\clw(X)$-valued Hardy space and
$M_z$ is the multiplication operator by the coordinate function $z$
on $H^2_{\clw(X)}(\D)$ (see Section \ref{sec-2}). 

On the other hand, a simpler (but complex enough) example of $n$-isometry can be obtained by taking restrictions of the $n$-shift $(M_{z_1}, \ldots, M_{z_n})$ on $H^2(\D^n)$ to an invariant subspace of $H^2(\D^n)$. Here $H^2(\D^n)$ denotes the Hardy space over the unit polydisc $\D^n$ and $(M_{z_1}, \ldots, M_{z_n})$ denotes the $n$-tuple of multiplication operators by the coordinate functions on $H^2(\D^n)$. A closed subspace $\cls$ of $H^2(\D^n)$ is said to be an \textit{invariant subspace} of $H^2(\D^n)$ if $M_{z_i} \cls \subseteq \cls$ for all $i = 1, \ldots, n$. Given an invariant subspace $\cls$ of $H^2(\D^n)$, we let
\[
R_{z_i} = P_{\cls} M_{z_i}|_{\cls} \in \clb(\cls) \quad \quad (i=1, \ldots, n),
\]
and denote by $\clt(\cls)$ the $C^*$-algebra generated by the commuting isometries $\{R_{z_1}, \ldots, R_{z_n}\}$. We simply say that $\clt(\cls)$ is the \textit{$C^*$-algebra corresponding to the invariant subspace} $\cls$. Clearly $\{R_{z_1}, \ldots, R_{z_n}\}$ is an $n$-isometry.

In this paper we aim to address three basic issues of $n$-isometries:
(i) analytic and canonical models for $n$-isometries, (ii) an abstract classification of joint invariant subspaces for $n$-isometries, and (iii) the nature of $C^*$-algebra $\clt(\cls)$ where $\cls$ is a finite codimensional invariant subspace in $H^2(\mathbb{D}^n)$. To that aim, for (i) and (ii), we consider the initial approach by Berger, Coburn and Lebow \cite{BCL} from a more modern point of view (due to Bercovici, Douglas and Foias \cite{BDF3}) along with the technique of \cite{MSS}. For (iii), we will examine Seto's approach \cite{MS} more closely from ``subspace'' approximation point of view.

We now briefly outline the setting and the main contributions of this paper. In \cite{BDF3}, motivated by Berger, Coburn and Lebow \cite{BCL},
Bercovici, Douglas and Foias proved the following result: An
$n$-isometry is unitarily equivalent to a model $n$-isometry. The
model $n$-isometries are defined as follows \cite{BDF3}: Consider a
Hilbert space $\cle$, unitary operators $\{U_1, \ldots, U_n\}$ on
$\cle$ and orthogonal projections $\{P_1, \ldots, P_n\}$ on $\cle$.
Let $\{\Phi_1, \ldots, \Phi_n\} \subseteq H^\infty_{\clb(\cle)}(\D)$
be bounded $\clb(\cle)$-valued holomorphic functions (polynomials)
on $\D$, where
\[
\Phi_i (z) = U_i (P_i^\perp + z P_i) \quad \quad \quad (z \in \D),
\]
and $i = 1, \ldots, n$. Then the $n$-tuple of multiplication
operators $(M_{\Phi_1}, \ldots, M_{\Phi_n})$ on $H^2_{\cle}(\D)$ is
called a \textit{model $n$-isometry} if the following conditions are
satisfied:

(a) $U_iU_j =U_j U_i$ for all $i, j = 1, \ldots n$;

(b) $U_1 \cdots U_n = I_{\cle}$;

(c) $P_i + U_i^* P_j U_i = P_j + U_j^* P_i U_j \leq I_{\cle}$ for
all $i\neq j$; and

(d) $P_1 + U_1^* P_2 U_1 + U_1^* U_2^* P_3 U_2 U_1 + \cdots + U_1^*
U_2^* \cdots U_{n-1}^* P_n U_{n-1}\cdots U_2 U_1 = I_{\cle}$.

\NI It is easy to see that a model $n$-isometry is also an
$n$-isometry (see page 643 in \cite{BDF3}).

Throughout the paper, given a Hilbert space $\clh$ and a closed
subspace $\cls$ of $\clh$, $P_{\cls}$ will denote the orthogonal
projection of $\clh$ onto $\cls$. We also set $P^\perp_{\cls} =
I_{\clh} - P_{\cls}$, that is,
\[
P_{\cls}^{\perp} = P_{\cls^\perp}.
\]
Recall that two $n$-tuples of commuting operators $(A_1, \ldots, A_n)$ on $\clh$ and $(B_1, \ldots, B_n)$ on $\clk$ are said to be \textit{unitarily
equivalent} if there exists a unitary operator $U : \clh \raro \clk$
such that $U A_i = B_i U$ for all $i = 1, \ldots, n$.

We refer to Bercovici, Douglas and Foias \cite{BDF1, BDF2, BDF3} and also
\cite{RGD1}, \cite{DY}, \cite{GS}, \cite{BKS}, \cite{DSS}, \cite{GG}, \cite{HQY}, \cite{MMSS}, \cite{DP}, \cite{MS} and \cite{RY1, RY2} for more on $n$-isometries, $n \geq 2$, and related topics.

Our first main result, Theorem \ref{th-char1}, states that an
$n$-isometry is unitarily equivalent to an explicit (and canonical)
model $n$-isometry. In other words, given an $n$-isometry $(V_1,
\ldots, V_n)$ on $\clh$, we explicitly solve the above conditions
(a)-(d) for some Hilbert space $\cle$, unitary operators $\{U_1,
\ldots, U_n\}$ on $\cle$ and orthogonal projections $\{P_1, \ldots,
P_n\}$ on $\cle$ so that the corresponding model $n$-isometry
$(M_{\Phi_1}, \ldots, M_{\Phi_n})$ is unitarily equivalent to $(V_1,
\ldots, V_n)$. This also gives a new proof of Bercovici, Douglas and
Foias theorem. On the one hand, our model $n$-isometry is explicit
and canonical. On the other hand, our proof is perhaps more
computational than the one in \cite{BDF3}. Another advantage of our
approach is the proof of a list of useful equalities related to
commuting isometries, which can be useful in other contexts.

Our second main result concerns a characterization of joint
invariant subspaces of model $n$-isometries.  To be precise, let
$\clw$ be a Hilbert space, and let $(M_{\Phi_1}, \ldots,
M_{\Phi_n})$ be a model $n$-isometry on $H^2_{\clw}(\D)$. Let $\cls$
be a closed subspace of $H^2_{\clw}(\D)$. In Theorem \ref{th-BLH N},
we prove that $\cls$ is invariant for $(M_{\Phi_1}, \ldots,
M_{\Phi_n})$ on $H^2_{\clw}(\D)$ if and only if there exist a
Hilbert space $\clw_*$, an inner function $\Theta \in
H^\infty_{\clb(\clw_*, \clw)}(\D)$ and a model $n$-isometry
$(M_{\Psi_1}, \ldots, M_{\Psi_n})$ on $H^2_{\clw_*}(\D)$ such that
\[
\cls = \Theta H^2_{\clw_*}(\D),
\]
and
\[
\Phi_i \Theta = \Theta \Psi_i,
\]
for all $ i = 1, \ldots, n$. Moreover, the above representation is
unique in an appropriate sense (see the remark following Theorem
\ref{th-BLH N}).

The third and final result concerns $C^*$-algebras corresponding to finite codimensional invariant subspaces in $H^2(\D^n)$. To be more specific, recall that if $n = 1$ and $\cls$ and $\cls'$ are invariant subspaces of $H^2(\D)$, then $U \clt(\cls) U^* = \clt(\cls')$ for some unitary $U : \cls \raro \cls'$. Indeed, since $\cls = \theta H^2(\D)$ for some inner function $\theta \in H^\infty(\D)$, it follows, by Beurling theorem, that $U: = M_{\theta} : H^2(\D) \raro \cls$ is a unitary and hence $U^* \clt(\cls) U = \clt(H^2(\D))$. Clearly, the general case follows from this special case. It is then natural to ask: If $n > 1$ and $\cls$ and $\cls'$ are submodules of $H^2(\D^n)$, are $\clt(\cls)$ and $\clt(\cls')$ isomorphic as $C^*$-algebras? We say that $\clt(\cls)$ and $\clt(\cls')$ are \textit{isomorphic} as $C^*$-algebras if $U \clt(\cls) U^* = \clt(\cls')$ holds for some unitary $U : \cls \raro \cls'$.

In the same paper \cite{BCL}, Berger, Coburn and Lebow asked whether $\clt(\cls)$ is isomorphic to $\clt(H^2(\D^2))$ for every  finite codimensional invariant subspaces $\cls$ in $H^2(\D^2)$. This question was recently answered positively by Seto in \cite{MS}. Here we extend Seto's answer from $H^2(\D^2)$ to the general case $H^2(\D^n)$, $n \geq 2$.

The rest of this paper is organized as follows. In Section 2 we study and review
the analytic construction of $n$-isometries. In Section 3 we study
more closely at the $n$-isometries and examine a (canonical) model
$n$-isometry. A characterization of invariant subspaces is
given in Section 4. Finally, in Section 5, we prove that $\clt(\cls)$ is isomorphic to $\clt(H^2(\D^n))$ where $\cls$ is a finite codimensional invariant subspaces in $H^2(\D^n)$.

\newsection{$n$-isometries}\label{sec-2}

In this section, following \cite{MSS}, we derive an explicit
analytic representation of $n$-isometries. For motivation, let us
recall that if $X$ on $\clh$ is an isometry, then $X$ is a shift
operator if and only if $X$ and $M_z$ on $H^2_{\clw(X)}(\D)$ are
unitarily equivalent. Here, $M_z$ denotes the multiplication operator
by the coordinate function $z$ on $H^2_{\clw(X)}(\D)$, that is, $(M_z f)(w) = w f(w)$ for all $f \in H^2_{\clw(X)}(\D)$ and $w \in \D$. Explicitly, if $X$
is a shift on $\clh$, then
\[
\clh = \mathop{\oplus}_{m=0}^\infty X^m \clw(X),
\]
where
\[
\clw(X) = \ker X^* = \clh \ominus X \clh,
\]
is the \textit{wandering subspace} for $X$ (see Halmos \cite{H}).
Hence the natural map $\Pi_X : \clh \raro H^2_{\clw(X)}(\D)$ defined
by
\[
\Pi_X (X^m \eta) = z^m \eta,
\]
for all $m \geq 0$ and $\eta\in \clw(X)$, is a unitary operator and
\[
\Pi_X X = M_z \Pi_X.
\]
We call $\Pi_X$ the \textit{Wold-von Neumann decomposition} of the
shift $X$.

Now let $\clh$ be a Hilbert space, and let $(V_1, \ldots, V_n)$ be
an $n$-isometry on $\clh$. \textsf{Throughout this paper, we shall
use the following notations}:
\[
V = \mathop{\Pi}_{k=1}^n V_k,
\]
and
\[
\tv_i = \mathop{\Pi}_{j \neq i} V_j,
\]
for all $i = 1, \ldots, n$. For simplicity, we also use the notation
\[
\clw = \clw(V),
\]
and
\[
\clw_i=\clw(V_i)\ \text{ and } \tw_i = \clw(\tv_i),
\]
for all $i = 1, \ldots, n$. With this tool, it is easy to see that
$\ker V_i^*, \ker \tilde{V}^*_i \subseteq \ker V^*$, that is,
\[
\clw_i, \tilde{\clw}_i \subseteq \clw,
\]
for all $i = 1, \ldots, n$. We denote by $P_{\clw_i}$ and
$P_{\tw_i}$ the orthogonal projections of $\clw$ onto the subspaces
$\clw_i$ and $\tw_i$, respectively. In particular, $P_{\clw_i},
P_{\tw_i} \in \clb(\clw)$ and
\[
P_{\tw_i} + P_{\tw_i}^\perp = I_{\clw},
\]
for all $i = 1, \ldots, n$.

Let $\Pi_V : \clh \raro H^2_{\clw}(\D)$ be the Wold-von Neumann
decomposition of $V$. Then
\[
\Pi_V V_i \Pi_V^* \in \{M_z\}',
\]
and hence there exists $\Phi_i \in H^{\infty}_{\clb(\clw)}(\D)$
\cite{NF} such that $\Pi_V V_i \Pi_V^* = M_{\Phi_i}$ or,
equivalently,
\[
\Pi_V V_i = M_{\Phi_i} \Pi_V,
\]
for all $i = 1, \ldots, n$. Note that $M_{\Phi_i}$ on
$H^2_{\clw}(\D)$ is defined by
\begin{equation}\label{eq-Mphi}
(M_{\Phi_i} f)(w) = {\Phi_i}(w) f(w),
\end{equation}
for all $f \in H^2_{\clw}(\D)$, $w \in \D$ and $i = 1, \ldots, n$.
We now proceed to compute the bounded analytic functions
$\{\Phi_i\}_{i=1}^n$. Our method follows the construction in
\cite{MSS}. In fact, a close variant of Theorem \ref{th-char1} below
follows from Theorems 3.4 and 3.5 of \cite{MSS}. We will only sketch
the construction, highlighting the essential ingredients for our
present purpose.

Let $i \in \{1, \ldots, n\}$, $w \in \D$ and $\eta \in \clw$. Then
from \eqref{eq-Mphi}, we have that
\[
\begin{split}
\Phi_i(w) \eta = (M_{\Phi_i} \eta)(w) = (\Pi_V V_i \Pi_V^* \eta)(w).
\end{split}
\]
Now it follows from the definition of $\Pi_V$ that $\Pi_V^* \eta =
\eta$, and hence $\Phi_i(w) \eta= (\Pi_V V_i \eta)(w)$. But $I_{\clw} = P_{\tw_i} + \tilde{V}_i \tilde{V}_i^*|_{\clw}$ yields that $V_i \eta = V_i P_{\tw_i} \eta + V \tilde{V}_i^* \eta$ and thus
\[
\begin{split}
\Pi_V V_i \eta & = \Pi_V(V_i P_{\tw_i} \eta + V \tilde{V}_i^* \eta)
\\
& = \Pi_V(V_i P_{\tw_i} \eta) + \Pi_V(V \tilde{V}_i^* \eta)
\\
& = \Pi_V(V_i P_{\tw_i} \eta) + M_z \Pi_V (\tilde{V}_i^* \eta),
\end{split}
\]
as $\Pi_V V = M_z \Pi_V$. Now, since $V^* (V_i (I - \tilde{V}_i \tilde{V}_i^*)V_i^*) = 0$ and $V^* (\tilde{V}_i^* \eta) = 0$, it follows that $V_i P_{\tw_i} \eta\in\clw$ and $\tilde{V}_i^* \eta
\in \clw$. This implies that
\[
\Pi_V V_i \eta = V_i P_{\tw_i} \eta + M_z \tilde{V}_i^* \eta,
\]
and so $\Phi_i(w) \eta = V_i P_{\tw_i} \eta + w \tilde{V}_i^* \eta$. It follows that $\Phi_i(w) = V_i|_{\tw_i} + w \tilde{V}_i^*|_{\tv_i \clw_i}$ as $\clw = \tv_i \clw_i \oplus \tw_i$. Finally, $\clw = \clw_i \oplus V_i \tw_i$ implies that
\[
U_i =
\begin{bmatrix}
\tv_i^*|_{\tv_i \clw_i} & 0\\
0 & V_i|_{\tw_i} \end{bmatrix} :
\begin{array}{c}
\tv_i \clw_i\\
\oplus\\ \tw_i
\end{array}
\raro
\begin{array}{c}
\clw_i \\
\oplus\\ V_i \tw_i
\end{array},
\]
is a unitary operator on $\clw$. Therefore
\[
\Phi_i(w) = U_i (P_{\tw_i} + w P_{\tw_i}^{\perp}),
\]
for all $w \in \D$. By definition of $U_i$, it follows that $U_i = (V_i P_{\tilde{{\clw}_i}} + {\tilde{V_i}}^*)|_{\clw}$. This and
\begin{equation}\label{eq-vt}
V_i P_{\tw_i} = P_{\clw} V_i,
\end{equation}
yields $U_i = (P_{\clw} V_i + {\tilde{V_i}}^*)|_{\clw}$. Summarizing the discussion above, we have the following:

\begin{thm}\label{th-char1}
Let $(V_1, \ldots, V_n)$ be an $n$-isometry on a Hilbert space
$\clh$, and let $V = \displaystyle{\Pi}_{i=1}^n V_i$. Let $\Pi_V : \clh \raro H^2_{\clw}(\D)$ be the Wold-von Neumann
decomposition of $V$. Then
\[
\Pi_V V_i = M_{\Phi_i} \Pi_V,
\]
where
\[
\Phi_i(z) =  U_i (P_{\tw_i} + z P_{\tw_i}^{\perp}),
\]
for all $z \in \D$, and
\[
U_i = (P_{\clw} V_i + {\tilde{V_i}}^*)|_{\clw}.
\]
is a unitary operator on $\clw$ and $i = 1, \ldots, n$.

In particular, $(V_1, \ldots, V_n)$ on $\clh$ and $(M_{\Phi_1}, \ldots,
M_{\Phi_n})$ on $H^2_{\clw}(\D)$ are unitarily equivalent.
\end{thm}

In the following section, we will explore the coefficients of
$\Phi_i$, $i = 1, \ldots, n$, in more details.

\newsection{Model $n$-isometries}

In this section, we propose a canonical model for $n$-isometries. We
study the coefficients of the one-variable polynomials in Theorem
\ref{th-char1} more closely and prove that the corresponding
$n$-isometry  $(M_{\Phi_1}, \ldots, M_{\Phi_n})$ on $H^2_{\clw}(\D)$
is a model $n$-isometry (see Section 1 for the definition of model
$n$-isometries).

Let $(V_1, \ldots, V_n)$ be an $n$-isometry on a Hilbert space
$\clh$. Consider the analytic representation $(M_{\Phi_1}, \ldots,
M_{\Phi_n})$ on $H^2_{\clw}(\D)$ of $(V_1, \ldots, V_n)$ as in
Theorem \ref{th-char1}. First we prove that $\{U_j\}_{j=1}^n$ is a
commutative family. Let $p, q \in \{1, \ldots, n\}$ and $p \neq q$.
As $\clw = \ker V^*$, it follows that
\[
\tv_p^* \tv_q^*|_{\clw} = 0.
\]
Then using (\ref{eq-vt}) we obtain
\[
\begin{split}
U_p U_q & = (P_{\clw} V_p + \tv_p^*)(P_{\clw} V_q + \tv_q^*)|_{\clw}
\\
& = (P_{\clw} V_p P_{\clw} V_q + \tv_p^* P_{\clw} V_q + P_{\clw} V_p
\tv_q^*)|_{\clw}
\\
& = (P_{\clw} V_p V_q + \mathop{\Pi}_{i \neq p,q} V_i^* P_{\tw_q} +
V_p P_{\tw_p} \tv_q^*)|_{\clw}
\\
& = (P_{\clw} V_p V_q + (\mathop{\Pi}_{i \neq p,q} V_i^*) (P_{\tw_q}
+ \tv_q P_{\tw_p} \tv_q^*))|_{\clw}
\\
& = (P_{\clw} V_p V_q + (\mathop{\Pi}_{i \neq p,q} V_i^*))|_{\clw},
\end{split}
\]
as $(P_{\tw_q} + \tv_q P_{\tw_p} \tv^*_q)|_{\clw} = I_{\clw}$, and
hence
\[
U_p U_q = U_q U_p,
\]
follows by symmetry. Now if $I \subseteq \{1, \ldots, n\}$, then the
same line of arguments as above yields
\begin{equation}\label{eq-pUi}
\mathop{\Pi}_{i \in I} U_i = (P_{\clw} (\mathop{\Pi}_{i \in I} V_i)
+ (\mathop{\Pi}_{i \in I^c} V_i^*))|_{\clw}.
\end{equation}
In particular, since $P_{\clw} V|_{\clw} = 0$, we have that
\[
\mathop{\Pi}_{i=1}^n U_i = I_{\clw}.
\]

The following lemma will be crucial in what follow.
\begin{lem}\label{le-UW}
 Fix $1 \leq j \leq n$. Let $I \subseteq \{1, \ldots, n\}$, and let
$j \notin I$. Then
\[
(\mathop{\Pi}_{i \in I} U_i^*) P^\perp_{\tw_j} (\mathop{\Pi}_{i \in I}
U_i)
= (\mathop{\Pi}_{i \in I^c \setminus \{j\}} V_i)
(\mathop{\Pi}_{i \in I^c \setminus \{j\}} V_i^*)|_{\clw}- (\mathop{\Pi}_{i \in I^c} V_i)
(\mathop{\Pi}_{i \in I^c} V_i^*)|_{\clw}.
\]
\end{lem}
\begin{proof}
Since $P_{\tw_j}=I_{\clw}-P_{\clw}\tv_j\tv_j^*|_{\clw}$, we have $P^\perp_{\tw_j} = P_{\clw}\tv_j \tv_j^*|_{\clw}=\tv_j
\tv_j^*|_{\clw}$. By once again using the fact that $V^*|_{\clw} = P_{\clw} V|_{\clw} = 0$, and by (\ref{eq-pUi}), one sees that
\[
\begin{split}
(\mathop{\Pi}_{i \in I} U_i^*) P^\perp_{\tw_j} (\mathop{\Pi}_{i \in I}
U_i) & = [(\mathop{\Pi}_{i \in
I} V_i^*) + P_{\clw}(\mathop{\Pi}_{i \in I^c} V_i)]  \tv_j \tv_j^* [P_{\clw} (\mathop{\Pi}_{i \in I} V_i) + (\mathop{\Pi}_{i
\in I^c} V_i^*)]|_{\clw}
\\
& = (\mathop{\Pi}_{i \in I^c\setminus\{j\}} V_i) \tv_j^* P_{\clw}
(\mathop{\Pi}_{i \in I} V_i)|_{\clw}
\\
&=  (\mathop{\Pi}_{i \in I^c\setminus\{j\}} V_i)\tv_j^*(I-VV^*)
(\mathop{\Pi}_{i \in I} V_i)|_{\clw}
\\
& =  (\mathop{\Pi}_{i \in I^c\setminus\{j\}} V_i)
(\mathop{\Pi}_{i \in I^c\setminus \{j\}} V_i^*)|_{\clw}- (\mathop{\Pi}_{i \in I^c} V_i)
(\mathop{\Pi}_{i \in I^c} V_i^*)|_{\clw}
\end{split}
\]
This completes the proof of the lemma.
\end{proof}

In particular, if $I=\{p\}$ and $j=q$, where $p,q\in\{1,\dots,n\}$
and $p\neq q$, then
\[
U_p^* P^\perp_{\tw_q} U_p = (\mathop{\Pi}_{i \neq p,q}
V_i)(\mathop{\Pi}_{i \neq p,q} V_i^*)|_{\clw}-
\tilde{V_p}\tilde{V_p}^*|_{\clw},
\]
hence
\[
\begin{split}
(P_{\tw_p}^\perp + U_p^* P_{\tw_q}^\perp U_p) & =P_{\clw}\tilde{V_p}\tilde{V_p}^*|_{\clw} +
(\mathop{\Pi}_{i \neq p,q} V_i)(\mathop{\Pi}_{i \neq p,q} V_i^*)|_{\clw}- P_{\clw}\tilde{V_p}\tilde{V_p}^*|_{\clw}
\\
&=  (\mathop{\Pi}_{i \neq p,q} V_i)
(\mathop{\Pi}_{i \neq p,q} V_i^*)|_{\clw}
\\
&\le I_{\clw}.
\end{split}
\]
Therefore by symmetry, we have
\[
(P_{\tw_p}^\perp + U_p^* P_{\tw_q}^\perp U_p) = (P_{\tw_q}^\perp +
U_q^* P_{\tw_p}^\perp U_q) \leq I_{\clw}.
\]
Finally, we let $I_j = \{1, \ldots, j-1\}$ for all $1 < j \leq n$
and $I_{n+1}=\{1,\dots,n\}$. Then Lemma \ref{le-UW} implies that for
$1<j\le n$,
\[
(\mathop{\Pi}_{i \in I_j} U_i) P^\perp_{\tw_j} (\mathop{\Pi}_{i \in
I_j} U_i^*) = [(\mathop{\Pi}_{i \in I_{j+1}^c} V_i) (\mathop{\Pi}_{i
\in I_{j+1}^c} V_i^*) - (\mathop{\Pi}_{i \in I_{j}^c} V_i)
(\mathop{\Pi}_{i \in I_{j}^c} V_i^*)]|_{\clw}.
\]
This and $P^\perp_{\tw_1} = \tv_{1} \tv^*_{1}|_{\clw}$ imply that
\[
P^\perp_{\tw_1} + U_1^* P^\perp_{\tw_2} U_1 + U_1^* U_2^*
P^\perp_{\tw_3} U_2 U_1 + \cdots + (\mathop{\Pi}_{i=1}^{n-1} U^*_i)
P^\perp_{\tw_n} (\mathop{\Pi}_{i=1}^{n-1} U_i) = I_{\clw}.
\]

We summarize the above as follows.

\begin{thm}\label{th-M2}
If $(V_1, \ldots, V_n)$ be an $n$-isometry on a Hilbert space
$\clh$, and let $U_1,\ldots,U_n$ be unitary operators as in Theorem \ref{th-char1}. Then

\textup{(a)} $U_pU_q =U_qU_p$ for $p, q= 1, \ldots n$,

\textup{(b)} $\prod_{p=1}^{n} U_p =I_{\clw}$,

\textup{(c)} $(P_{\tw_i}^\perp + U_i^* P_{\tw_j}^\perp U_i) =
(P_{\tw_j}^\perp + U_j^* P_{\tw_i}^\perp U_j) \leq I_{\clw}$ $(1\le i< j\le n)$,

\textup{(d)} $P^\perp_{\tw_1} + U_1^* P^\perp_{\tw_2} U_1 + U_1^* U_2^*
P^\perp_{\tw_2} U_2 U_1 + \cdots + (\mathop{\Pi}_{i=1}^{n-1} U^*_i)
P^\perp_{\tw_n} (\mathop{\Pi}_{i=1}^{n-1} U_i) = I_{\clw}$.
\end{thm}

As a corollary, we have:

\begin{cor}\label{cor-MT2}
Let $\clh$ be a Hilbert space and $(V_1, \ldots, V_n)$ be an
$n$-isometry on $\clh$. Let $(M_{\Phi_1}, \ldots, M_{\Phi_n})$ be
the $n$-isometry as constructed in Theorem \ref{th-char1}, and let
$(M_{\Psi_1}, \ldots, M_{\Psi_n})$ on $H^2_{\tw}(\D)$, for some
Hilbert space $\tw$, unitary operators $\{\tilde{U}_i\}_{i=1}^n$ and
orthogonal projections $\{P_i\}_{i=1}^n$ on $\tw$, be a model
$n$-isometry. Then:

\textup{(a)} $(M_{\Phi_1}, \ldots, M_{\Phi_n})$ is a model $n$-isometry.

\textup{(b)} $(V_1, \ldots, V_n)$ and $(M_{\Phi_1}, \ldots, M_{\Phi_n})$ are
unitarily equivalent.

\textup{(c)} $(V_1, \ldots, V_n)$ and $(M_{\Psi_1}, \ldots, M_{\Psi_n})$ are
unitarily equivalent if and only if there exists a unitary operator
$W : \clw \raro \tw$ such that $W U_i = \tilde{U}_i W$ and $W P_i =
\tilde{P}_i W$ for all $i = 1,  \ldots, n$.
\end{cor}
\begin{proof}
Parts (a) and (b) follows directly from the previous theorem. The
third part is easy and readily follows from Theorem 4.1 in
\cite{MSS} or Theorem 2.9 in \cite{BDF3}.
\end{proof}

Combining Corollary \ref{cor-MT2} with Theorem \ref{th-M2}, we have the
following characterization of commutative isometric factors of shift
operators.

\begin{cor}\label{co-factor}
Let $\cle$ be a Hilbert space, and let $\{\Phi_i\}_{i=1}^n \subseteq
H^\infty_{\clb(\cle)}(\D)$ be a commutative family of isometric
multipliers. Then
\[
M_z = \mathop{\Pi}_{i = 1}^n M_{\Phi_j},
\]
or, equivalently
\[
\mathop{\Pi}_{i = 1}^n {\Phi_j} = z I_{\cle},
\]
if and only if, up to unitary equivalence, $(M_{\Phi_1}, \ldots,
M_{\Phi_n})$ is a model $n$-isometry.
\end{cor}

In other words, $z I_{\cle}$ factors as $n$ commuting isometric
multipliers $\{\Phi_i\}_{i=1}^n \subseteq H^\infty_{\clb(\cle)}(\D)$
if and only if there exist unitary operators $\{U_i\}_{i=1}^n$ on
$\cle$ and orthogonal projections $\{P_i\}_{i=1}^n$ on $\cle$
satisfying the properties (a) - (d) in Theorem \ref{th-M2} such that
$\Phi_i(z) = U_i(P_i^\perp + z P_i)$ for all $i=1, \ldots, n$.

\newsection{Joint Invariant Subspaces}

Let $\clw$ be a Hilbert space. Let $(M_{\Phi_1}, \ldots,
M_{\Phi_n})$ be a model $n$-isometry on $H^2_{\clw}(\D)$, and let
$\cls$ be a closed invariant subspace for $(M_{\Phi_1}, \ldots,
M_{\Phi_n})$ on $H^2_{\clw}(\D)$, that is
\[
M_{\Phi_i} \cls \subseteq \cls,
\]
for all $i = 1, \ldots, n$. Then $(M_{\Phi_1}|_{\cls}, \ldots,
M_{\Phi_n}|_{\cls})$ is an $n$-tuple of commuting isometries on
$\cls$. Clearly
\[
\mathop{\Pi}_{i=1}^n (M_{\Phi_i}|_{\cls}) = (\mathop{\Pi}_{i=1}^n
M_{\Phi_i})|_{\cls},
\]
and since
\[
\mathop{\Pi}_{j=1}^n M_{\Phi_j} = M_z,
\]
it follows that
\begin{equation}\label{MphiS}
(\mathop{\Pi}_{i=1}^n M_{\Phi_i})|_{\cls} = M_z|_{\cls},
\end{equation}
that is, $\cls$ is a invariant subspace for $M_z$ on
$H^2_{\clw}(\D)$. Moreover, since $M_z|_{\cls}$ is a shift on
$\cls$, the tuple $(M_{\Phi_1}|_{\cls}, \ldots, M_{\Phi_n}|_{\cls})$
is an $n$-isometry on $\cls$. Then by Corollary
\ref{cor-MT2} there is a model $n$-isometry $(M_{\Psi_1}, \ldots,
M_{\Psi_n})$ on $H^2_{\tw}(\D)$, for some Hilbert space $\tw$, such
that $(M_{\Phi_1}|_{\cls}, \ldots, M_{\Phi_n}|_{\cls})$ and $(M_{\Psi_1}, \ldots,
M_{\Psi_n})$ are unitarily equivalent. The main purpose of this
section is to describe the invariant subspaces $\cls$
in terms of the model $n$-isometry $(M_{\Psi_1}, \ldots, M_{\Psi_n})$.

As a motivational example, consider the classical $n = 1$ case. Here
the model $1$-isometry is the shift operator $M_z$ on
$H^2_{\clw}(\D)$ for some Hilbert space $\clw$. Let $\cls$ be a
closed subspace of $H^2_{\clw}(\D)$. Then by the Beurling \cite{B},
Lax \cite{L} and Halmos \cite{H} theorem (or see page 239, Theorem
2.1 in \cite{FF}), $\cls$ is invariant for $M_z$ if and only if
there exist a Hilbert space $\clw_*$ and an inner function $\Theta
\in H^\infty_{\clb(\clw_*, \clw)}(\D)$ such that
\[
\cls = \Theta H^2_{\clw_*}(\D).
\]
Moreover, in this case, if we set
\[
V = M_z|_{\cls},
\]
then $\clw_* = \cls \ominus z \cls$ and $V$ on $\cls$ and $M_z$ on
$H^2_{\clw_*}(\D)$ are unitarily equivalent. This follows directly
from the above representation of $\cls$. Indeed, it follows that $X
= M_{\Theta} : H^2_{\clw_*}(\D) \raro \mbox{ran} M_{\Theta} = \cls$
is a unitary operator and
\[
X M_z = V X.
\]

Turning to the case $n > 1$, let $(M_{\Phi_1}, \ldots, M_{\Phi_n})$
be a model $n$-isometry on $H^2_{\clw}(\D)$, and let $\cls$ be a
closed invariant subspace for $(M_{\Phi_1}, \ldots, M_{\Phi_n})$ on
$H^2_{\clw}(\D)$. Let
\[
\clw_* = \cls \ominus z \cls.
\]
Since $\cls$ is an invariant subspace for $M_z$ on $H^2_{\clw}(\D)$
(see Equation \eqref{MphiS}), by Beurling, Lax and Halmos theorem,
there exists an inner function $\Theta \in H^\infty_{\clb(\clw_*,
\clw)}(\D)$ such that $\cls$ can be represented as
\[
\cls = \Theta H^2_{\clw_*}(\D),
\]
If $1 \leq j \leq n$, then
\[
\Phi_j \cls \subseteq \cls,
\]
implies that $\mbox{ran~} (M_{\Phi_j} M_{\Theta}) \subseteq \mbox{ran~}
M_{\Theta}$, and so by Douglas's range and inclusion theorem \cite{RGD}
\[
M_{\Phi_j} M_{\Theta} = M_{\Theta} M_{\Psi_j},
\]
for some $\Psi_j \in H^{\infty}_{\clb(\clw_*)}(\D)$. Note that
$M_{\Phi_j} M_{\Theta}$ is an isometry and $ \|\Theta \Psi_j f\| =
\|\Psi_j f\|$ for each  $f \in H^2_{\clw_*}(\D)$. But then $\|M_{\Psi_j} f\| = \|f\|$ implies that $M_{\Psi_j}$ is an isometry, that is, $\Psi_j$
is an inner function, and hence
\[
M_{\Psi_j} = M_{\Theta}^* M_{\Phi_j} M_{\Theta},
\]
for all $j = 1, \ldots, n$. So
\[
\mathop{\Pi}_{i=1}^n M_{\Psi_i} = (M_{\Theta}^* M_{\Phi_1}
M_{\Theta}) \cdots (M_{\Theta}^* M_{\Phi_n} M_{\Theta}).
\]
Now $P_{\text{ran~} M_{\Theta}} = M_{\Theta} M_{\Theta}^*$ and
${\Phi_j} \Theta H^2_{\clw_*}(\D) \subseteq\Theta H^2_{\clw_*}(\D)$
implies that
\[
M_{\Theta} M_{\Theta}^* M_{\Phi_j} M_{\Theta} = M_{\Phi_j}
M_{\Theta},
\]
for all $j = 1, \ldots, n$. Consequently
\[
\mathop{\Pi}_{j=1}^n M_{\Psi_j} = M_{\Theta}^*
(\mathop{\Pi}_{j=1}^n M_{\Phi_j}) M_{\Theta}^* = M_{\Theta}^* M_z M_{\Theta} = M_{\Theta}^* M_{\Theta} M_z = M_z,
\]
that is, $(M_{\Psi_1}, \ldots, M_{\Psi_n})$ is an $n$-isometry on
$H^2_{\clw_*}(\D)$. In view of Corollary \ref{co-factor}, this also
implies that the tuple $(M_{\Psi_1}, \ldots, M_{\Psi_n})$ is a model
$n$-isometry. Therefore, we have the following theorem:

\begin{thm}\label{th-BLH N}
Let $\clw$ be a Hilbert space. Let $(M_{\Phi_1}, \ldots,
M_{\Phi_n})$ be a model $n$-isometry on $H^2_{\clw}(\D)$, and let
$\cls$ be a closed subspace of $H^2_{\clw}(\D)$. Then $\cls$ is
invariant for $(M_{\Phi_1}, \ldots, M_{\Phi_n})$ on $H^2_{\clw}(\D)$
if and only if there exist a Hilbert space $\clw_*$, an inner
function $\Theta \in H^\infty_{\clb(\clw_*, \clw)}(\D)$ and a model
$n$-isometry $(M_{\Psi_1}, \ldots, M_{\Psi_n})$ on
$H^2_{\clw_*}(\D)$ such that
\[
\cls = \Theta H^2_{\clw_*}(\D),
\]
and
\[
\Phi_j \Theta = \Theta \Psi_j,
\]
for all $j = 1, \ldots, n$.
\end{thm}

The representation of $\cls$ is unique in the following sense: if
there exist a Hilbert space $\hat{\clw}$, an inner multiplier
$\hat{\Theta} \in H^\infty_{\clb(\hat{\clw}, \clw)}(\D)$ and a model
$n$-isometry $( M_{\hat{\Psi}_1}, \ldots, M_{ \hat{\Psi}_n})$ on
$H^2_{\hat{\clw}}(\D)$ such that $\cls = \hat{\Theta} H^2_{\hat{\clw}}(\D)$ and $\Phi_i \hat{\Theta} = \hat{\Theta} \hat{\Psi}_i$
for all $i = 1, \ldots, n$, then there exists a unitary $\tau :
\clw_* \raro \hat{\clw}$ such that
\[
\Theta = \hat{\Theta} \tau,
\]
and
\[
\hat{\Psi}_j\tau = \tau \Psi_j,
\]
for all $j = 1, \ldots, n$. In other words, the model $n$-isometries
$( M_{\hat{\Psi}_1}, \ldots, M_{ \hat{\Psi}_n})$ on
$H^2_{\hat{\clw}}(\D)$ and  $(M_{\Psi_1}, \ldots, M_{\Psi_n})$ on
$H^2_{\clw_*}(\D)$ are unitary equivalent (under the same unitary
$\tau$). Indeed, the existence of the unitary $\tau$ along with the
first equality follows from the uniqueness of the Beurling, Lax and
Halmos theorem (cf. page 239, Theorem 2.1 in \cite{FF}). For the
second equality, observe that (see the uniqueness part in
\cite{MMSS})
\[
\hat{\Theta} \tau \Psi_i = \Theta \Psi_i = \Phi_i \Theta = \Phi_i \hat{\Theta} \tau,
\]
that is $\hat{\Theta} \tau \Psi_i = \hat{\Theta} \hat{\Psi}_i\tau$, and so
\[
\tau \Psi_i =\hat{\Psi}_i \tau,
\]
for all $i=1,\ldots,n$.

It is curious to note that the content of Theorem \ref{th-BLH N} is
related to the question \cite{AY} and its answer \cite{JS3} on the
classifications of invariant subspaces of $\Gamma$-isometries. A
similar result also holds for invariant subspaces for the
multiplication operator tuple on the Hardy space over the unit
polydisc in $\mathbb{C}^n$ (see \cite{MMSS}).

Our approach to $n$-isometries has other applications to $n$-tuples,
$n \geq 2$, of commuting contractions (cf. see \cite{DSS}) that we
will explore in a future paper.

\newsection{$C^*$-algebras generated by commuting isometries}

In this section, we extend Seto's result \cite{MS} on isomorphic $C^*$-algebras of invariant subspaces of finite codimension in $H^2(\D^2)$ to that in $H^2(\D^n)$, $n \geq 2$. Given a Hilbert space $\clh$, the set of all compact operators from $\clh$ to itself is denoted by $K(\clh)$. Recall that, for a closed subspace $\cls \subseteq H^2(\D^n)$, we say that $\cls$ is an invariant subspace of $H^2(\D^n)$ if $M_{z_i} \cls \subseteq \cls$ for all $i = 1, \ldots, n$. Also recall that in the case of an invariant subspace $\cls$ of $H^2(\D^n)$, $(R_{z_1}, \ldots, R_{z_n})$ is an $n$-isometry on $\cls$ where
\[
R_{z_i} = M_{z_i}|_{\cls} \in \clb(\cls) \quad \quad (i=1, \ldots, n).
\]

\begin{lem}\label{lemma-K}
If $\cls$ is an invariant subspace of finite codimension in $H^2(\D^n)$,
then $K(\cls) \subseteq \mathcal T(\cls)$.
\end{lem}

\begin{proof}
Since $\mathcal T (\cls)$ is an irreducible $C^*$-algebra (cf. \cite{MS}), it is enough
to prove that $\mathcal T (\cls)$ contains a non-zero compact operator.
As
\[
\mathop{\Pi}_{i=1}^n (I_{H^2(\D^n)} - M_{z_i}M_{z_i}^*) = P_{\mathbb C}\in \clt(H^2(\D^n)),
\]
we are done when $\cls = H^2(\D^n)$. Let us now suppose that $\cls$ is a proper subspace of $H^2(\D^n)$. For arbitrary $1\le i<j\le n$, we have
\[
[R_{z_i}^*, R_{z_j}] = P_{\cls} M_{z_j} P_{\cls^{\perp}} M_{z_i}^*|_{\cls} \in K(\cls),
\]
as $\cls^{\perp}$ is finite dimensional. It remains for us to prove that $[R_{z_i}^*, R_{z_j}] \neq 0$ for some $1\le i<j\le n$. If not, then $\cls$ is a proper doubly commuting invariant subspace with finite codimension. As a result, we would have $\cls = \vp H^2(\D^n)$ for some inner function $\vp \in H^\infty(\D^n)$ and hence $\cls$ has infinite codimension (see the corollary in page 969, \cite{AC}), a contradiction.
\end{proof}

Given a Hilbert space $\clh$ and nested closed subspaces $\clm_1 \subseteq \clm_2 \subseteq \clh$, the orthogonal projection of $\clm_2$ onto $\clm_1$ will be denoted by $P^{\clm_2}_{\clm_1}$. Note that $P^{\clm_2}_{\clm_1} \in \clb(\clm_2)$. When $\clm_2$ is clear from context we shall simply write $P_{\clm_1}$ for $P^{\clm_2}_{\clm_1}$. If, in addition, $\clh = H^2(\D^n)$ and $\clm_1$ is an invariant subspaces of $H^2(\D^n)$, then we set
\[
R_{z_i}^{\clm_1} = M_{z_i}|_{\clm_1} \in \clb(\clm_1),
\]
and simply write $R_{z_i}$, $i=1, \ldots, n$, when $\clm_1$ is clear from the context. Also, in what follows, a finite rank operator on a Hilbert space will be denoted by $F$ (without referring to the ambient Hilbert space).

\begin{lem}\label{lemma-M1 and M2}
Suppose $\clm_1$ and $\clm_2$ are invariant subspaces of $H^2(\D^n)$, $\clm_1 \subseteq \clm_2$ and $\mbox{dim}(\clm_2 \ominus \clm_1) < \infty$. Then $\clt(\clm_1) = \{P_{\clm_1} T|_{\clm_1}: T \in \clt(\clm_2)\}$. Moreover, if $\cll$ is a closed subspace of $\clm_1$ and $P^{\clm_2}_{\cll} \in \clt(\clm_2)$, then $P^{\clm_1}_{\cll} \in \clt(\clm_1)$.
\end{lem}
\begin{proof}
Note that $P_{\clm_1}R_{z_i}^{\clm_2}|_{\clm_1}= R_{z_i}^{\clm_1}$ and so, by taking adjoint,
we have
\[
P_{\clm_1}(R_{z_i}^{\clm_2})^*|_{\clm_1}= (R_{z_i}^{\clm_1})^*,
\]
for all $i = 1, \ldots, n$. Then $R_{z_i}^{\clm_1}(R_{z_j}^{\clm_1})^*
= P_{\clm_1}R_{z_i}^{\clm_2}P_{\clm_1}^{\clm_2}(R_{z_j}^{\clm_2})^*|_{\clm_1}$, $i = 1, \ldots, n$. This yields
\begin{align*}
R_{z_i}^{\clm_1}(R_{z_j}^{\clm_1})^*
& =P_{\clm_1}R_{z_i}^{\clm_2} I_{\clm_2} (R_{z_j}^{\clm_2})^*|_{\clm_1} -
P_{\clm_1} R_{z_i}^{\clm_2} P_{\clm_2 \ominus \clm_1}^{\clm_2} (R_{z_i}^{\clm_2})^*|_{\clm_1}
\\
&=P_{\clm_1}R_{z_i}^{\clm_2}(R_{z_j}^{\clm_2})^*|_{\clm_1} + F,
\end{align*}
for all $i, j = 1, \ldots, n$, as $\mbox{dim} (\clm_2 \ominus \clm_1) < \infty$. Similarly $(R_{z_j}^{\clm_1})^* R_{z_i}^{\clm_1} =
P_{\clm_1} (R_{z_j}^{\clm_2})^* R_{z_i}^{\clm_2}|_{\clm_1} + F$ for all $i, j = 1, \ldots, n$. Now let $T_1 \in \clt(\clm_1)$ be a finite word formed from the symbols
\[
\{R_{z_i}^{\clm_1}, (R_{z_i}^{\clm_1})^*: i =1, \ldots, n\},
\]
and let $T_2 \in \clt(\clm_2)$ be the same word but formed from the
corresponding symbols in
\[
\{R_{z_i}^{\clm_2}, (R_{z_i}^{\clm_2})^*: i =1, \ldots, n\}.
\]
Then $T_1 = P_{\clm_1} T_2|_{\clm_1} + F$. Since both $\clt(\clm_1)$ and
$\{P_{\clm_1} T|_{\clm_1}: T \in \clt(\clm_2)\}$ are closed subspaces
of $\clb(\clm_1)$ and both contain all the compact operators in $\clb(\clm_1)$, it follows that $\clt(\clm_1) = \{P_{\clm_1} T|_{\clm_1}: T \in \clt(\clm_2)\}$. The second assertion now clearly follows from the first one.
\end{proof}

A thorough understanding of co-doubly commuting invariant subspaces of finite codimension is important to analyze $C^*$-algebras of invariant subspaces of finite codimension in $H^2(\D^n)$. If $\cls$ is a closed invariant subspace of $H^2(\D)$, then we know that $\cls = \theta H^2(\D)$ for some inner function $\theta \in H^\infty(\D)$. To simplify notations, for a given inner function $\theta \in H^\infty(\D)$, we denote
\[
\cls_{\theta} = \theta H^2(\D), \quad \mbox{and} \quad \clq_{\theta} = H^2(\D) \ominus \theta H^2(\D).
\]
Also, given an inner function $\theta_i \in H^\infty(\D)$, $1 \leq i \leq n$, denote by $M_{\theta_i}$ the multiplication operator
\[
(M_{\theta_i} f)(z_1, \ldots, z_n) = \theta_i(z_i) f(z_1, \ldots, z_n)
\]
for all $f \in H^2(\D^n)$ and $(z_1, \ldots, z_n) \in \D^n$. Recall now that an invariant subspace $\cls$ of $H^2(\D^n)$ is said to be \textit{co-doubly commuting} \cite{Sar} if $\cls = \cls_{\Phi}$ where
\begin{equation}\label{eq-Def S Phi}
\cls_{\Phi} =(\clq_{\vp_1}\otimes\cdots\otimes \clq_{\vp_n})^{\perp},
\end{equation}
and $\vp_i$, $i = 1, \ldots, n$, is either inner or the zero function. Here, in view of \eqref{eq-Def S Phi} (or see \cite{Sar}), we have
\[
(M_{\vp_p} M_{\vp_p}^*) (M_{\vp_q} M_{\vp_q}^*) = (M_{\vp_q} M_{\vp_q}^*) (M_{\vp_p} M_{\vp_p}^*),
\]
for all $p,q = 1, \ldots, n$, and
\begin{equation}\label{eq-Projection S}
P_{\cls_{\Phi}} = I_{H^2(\D^n)} - \mathop{\Pi}_{i=1}^n (I_{H^2(\D^n)} - M_{\vp_i} M_{\vp_i}^*).
\end{equation}
It also follows that
\[
\cls_{\Phi} = M_{\vp_1} H^2(\D^n) + \cdots + M_{\vp_n} H^2(\D^n).
\]
Therefore, $\cls_{\Phi}$ has finite codimension if and only if $\vp_i$ is a finite Blashcke product for all $i=1, \ldots, n$. Moreover, if $\cls$ is an invariant subspace of $H^2(\D^n)$, then $\cls$ is of finite codimensional if and only if (cf. Lemma 3.1, \cite{MS}) there exist finite Blaschke products $\vp_1,\dots, \vp_n$ such that
\[
\cls_{\Phi} \subseteq \cls.
\]
Given $\cls_{\Phi}$ as in \eqref{eq-Def S Phi} and $1\le i<j\le n$, we define $\clq_{\Phi}{[i,j]}$ by
\[
\clq_{\Phi}{[i,j]} = \clq_{\vp_i}\otimes\clq_{\vp_{i+1}}\otimes\cdots\otimes\clq_{\vp_j} \subseteq H^2(\D^{j-i+1}).
\]
For notational simplicity, we set
\[
\cll_1 = {\clq_{\Phi}{[1, n-1]}}^\perp \otimes H^2(\D), \quad \cll_2 = {\clq_{\Phi}{[1, n-1]}} \otimes \cls_{\vp_n}, \quad \cll_3 = {\clq_{\Phi}{[1, n-1]}} \otimes H^2(\D),
\]
and
\[
\cll_2' = {\clq_{\Phi}[1, n-1]} \otimes \vp_n \cls_{\vp_n} \quad \mbox{and} \quad \cll_2'' = {\clq_{\Phi}[1, n-1]} \otimes \vp_n \clq_{\vp_n}.
\]
Clearly
\[
\cls_{\Phi} = \cll_1 \oplus \cll_2, \quad H^2(\D^n) = \cll_1 \oplus \cll_3,
\]
and
\[
\cll_2 = \cll_2' \oplus \cll_2''.
\]
\textsf{In what follows, given a closed subspace $\clm \subseteq H^2(\D^n)$, we will simply denote the orthogonal projection $P^{H^2(\D^n)}_{\clm}$ by $P_{\clm}$.}

\begin{lem}\label{projection}
If $\{\vp_i\}_{i=1}^n$ are finite Blashcke products, then $P_{\cll_1}, P_{\cll_2}, P_{\cll_2'}$ and $P_{\cll_2''}$ are in $\mathcal T (H^2(\D^n))$ and $P_{\cll_1}^{\cls_{\Phi}}, P_{\cll_2}^{\cls_{\Phi}}, P_{\cll_2'}^{\cls_{\Phi}}$ and $P_{\cll_2''}^{\cls_{\Phi}}$ are in $\mathcal T (\cls_{\Phi})$.
\end{lem}
\begin{proof}
By virtue of Lemma \ref{lemma-M1 and M2}, we only prove the lemma for $H^2(\D^n)$. Since $\cll_2''$ is finite-dimensional, it follows, by Lemma \ref{lemma-K}, that $P_{\cll_2''} \in \clt(H^2(\D^n))$. Since $\vp_1 \in H^\infty(\D)$ is a finite Blaschke product, it follows that $M_{\vp_i} = \vp_i(M_{z_i})$ for all $i = 1, \ldots, n$, and then, by \eqref{eq-Projection S}, $P_{\cls_{\Phi}} \in \clt(H^2(\D^n))$. In view of $\cls_{\Phi} = \cll_1 \oplus \cll_2$, it is then enough to prove only that $P_{\cll_2} \in \clt(H^2(\D^n))$. This readily follows from the equality
\[
P_{\cll_2} = \Big(\mathop{\Pi}_{i=1}^{n-1} (I_{H^2(\D^n)} - M_{\vp_i} M_{\vp_i}^*)\Big) M_{\vp_n} M_{\vp_n}^*.
\]
This completes the proof of the lemma.
\end{proof}

In particular, $\clt(\cls_{\Phi})$ contains a wealth of orthogonal projections. This leads to some further observations concerning the $C^*$-algebra $\clt(\cls_{\Phi})$. First, given $\cls_{\Phi}$ as in \eqref{eq-Def S Phi}, we consider the unitary operator $U: H^2(\D^n) \raro \cls_{\Phi}$ defined by
\[
U = \begin{bmatrix} I_{\cll_1} & 0 \\ 0 & M_{\vp_n} \end{bmatrix} : \begin{matrix} \cll_1 \\ \oplus \\ \cll_3\end{matrix} \raro \begin{matrix} \cll_1 \\ \oplus \\ \cll_2 \end{matrix}.
\]
Then $U = P_{\cll_1} + M_{\vp_n} P_{\cll_3}$ and
$U^* = P^{\cls_{\Phi}}_{\cll_1} + M_{\vp_n}^* P^{\cls_{\Phi}}_{\cll_2}$. We have the following result:

\begin{thm}\label{main1}
If $\{\vp_i\}_{i=1}^n$ are finite Blashcke products, then $U^*\clt(\cls_{\Phi}) U = \mathcal T(H^2(\D^n))$.
In particular, $\clt(\cls_{\Phi})$ and $\mathcal T(H^2(\D^n))$ are unitarily equivalent.
\end{thm}
\begin{proof}
A simple computation first confirms that
\[
U^* R_{z_n} U= M_{z_n}\in \mathcal T(H^2(\D^n)),
\]
that is, $M_{z_n} \in U^*\clt(\cls_{\Phi})U$ and $R_{z_n} \in U\clt(H^2(\D^n))U^*$.
Next, let $i = 1, \ldots, n-1$. Then $R_{z_i} U = M_{z_i} P_{\cll_1} + R_{z_i} M_{\vp_n} P_{\cll_3}=M_{z_i} P_{\cll_1} + M_{z_i} M_{\vp_n} P_{\cll_3}$ as $M_{\vp_n} \cll_3 = \cll_2 \subseteq \cls_{\Phi}$ and
so
\begin{align*}
U^*R_{z_i} U & = (P^{\cls_{\Phi}}_{\cll_1} + M_{\vp_n}^* P^{\cls_{\Phi}}_{\cll_2})
(M_{z_i} P_{\cll_1} + M_{z_i} M_{\vp_n} P_{\cll_3})
\\
& =  M_{z_i}P_{\cll_1} +
P_{\cll_1} M_{z_i} M_{\vp_n} P_{\cll_3} +
M_{\vp_n}^* P_{\cll_2}M_{z_i} M_{\vp_n} P_{\cll_3},
\end{align*}
as $M_{z_i}\cll_1\subseteq \cll_1$ and $M_{z_i}M_{\vp_n} \cll_3 = M_{z_i}\cll_2 \subseteq \cls_{\Phi}$. Then $U^*R_{z_i} U\in \mathcal T(H^2(\D^n))$ for al $i = 1, \ldots, n$, by Lemma \ref{projection}. In particular
\[
U^*\mathcal T (\cls_{\Phi}) U\subseteq \mathcal T (H^2(\D^n)).
\]
On the other hand, since $\cll_2 = \cll_2' \oplus \cll_2''$ and $\cll_2''$ is finite dimensional, it follows that $P_{\cll_2} = P_{\cll'_2} + F$, and thus $U^* = U^*|_{\cll_1} + U^*|_{\cll_2'} + F$. Now $U M_{z_i} U^*|_{\cll_1} = U M_{z_i}|_{\cll_1} = M_{z_i}|_{\cll_1}$ as $z_i \cll_1 \subseteq \cll_1$ and hence
\[
U M_{z_i} U^*|_{\cll_1} = R_{z_i}|_{\cll_1},
\]
and on the other hand
\[
\begin{split}
U M_{z_i} U^*|_{\cll'_2} = U(M_{z_i} M_{\vp_n}^*|_{\cll'_2})
= U(M_{z_i} P_{\cls_{\Phi}} M_{\vp_n}^*)|_{\cll'_2} = U (R_{z_i} R_{\vp_n}^*)|_{\cll'_2},
\end{split}
\]
where $R_{\vp_n} = M_{\vp_n}|_{\cls_{\Phi}}$. Moreover, since $\cll_3 = \cll_2 \oplus \cls_{\Phi}^\perp$ and $\cls_{\Phi}^\perp$ is finite dimensional, it follows that $P_{\cll_3} = P_{\cll_2} + F$, and thus
\[
\begin{split}
U M_{z_i} U^*|_{\cll'_2} & = P_{\cll_1} R_{z_i} R_{\vp_n}^*|_{\cll'_2} + M_{\vp_n} P_{\cll_3} R_{z_i} R_{\vp_n}^*|_{\cll'_2}
\\
& = P_{\cll_1} R_{z_i} R_{\vp_n}^*|_{\cll'_2} + M_{\vp_n} P_{\cll_2} R_{z_i} R_{\vp_n}^*|_{\cll'_2} + F
\\
& = P^{\cls_{\Phi}}_{\cll_1} R_{z_i} R_{\vp_n}^*|_{\cll'_2} +
R_{\vp_n} P^{\cls_{\Phi}}_{\cll_2} R_{z_i} R_{\vp_n}^*|_{\cll'_2} + F,
\end{split}
\]
and hence 
\[
U M_{z_i}U^* =  R_{z_i}P^{\cls_{\Phi}}_{\cll_1} +
P^{\cls_{\Phi}}_{\cll_1} R_{z_i} R_{\vp_n}^*P^{\cls_{\Phi}}_{\cll'_2}+
R_{\vp_n} P^{\cls_{\Phi}}_{\cll_2} R_{z_i} R_{\vp_n}^*P^{\cls_{\Phi}}_{\cll'_2}
+ F. 
\]
By Lemma \ref{projection}, it follows then that $U M_{z_i}U^* \in \clt(\cls_{\Phi})$ and so $U\mathcal T(H^2(\D^n))U^* \subseteq \mathcal T(\cls_{\Phi})$. Therefore, the conclusion follows from the fact that $U^* R_{z_n} U= M_{z_n}\in \mathcal T(H^2(\D^n))$.
\end{proof}

Now let $\cls$ be an invariant subspace of finite codimension, and let $\cls_{\Phi} \subseteq \cls$, as in \eqref{eq-Def S Phi}, for some finite Blashcke products $\{\vp_i\}_{i=1}^n$. We proceed to prove that $\clt(\cls)$ is unitarily equivalent to $\clt(\cls_{\Phi})$. Let
\[
m :=\dim (\cls\ominus \cls_{\Phi}).
\]
Observe that
\[
P_{\cls_{\Phi}} = M_{\vp_1} M_{\vp_1}^* + (I_{H^2(\D^n)} - M_{\vp_1} M_{\vp_1}^*) \Big(I_{H^2(\D^n)} - \mathop{\Pi}_{i=2}^n (I_{H^2(\D^n)} - M_{\vp_i} M_{\vp_i}^*) \Big),
\]
and so
\[
\cls_{\Phi} = \Big(\cls_{\vp_1}\otimes H^2(\D^{n-1}) \Big) \oplus \Big(\clq_{\vp_1}\otimes {\clq_{\Phi}[2,n]}^{\perp}\Big).
\]

\begin{lem}\label{projection 1}
$P_{\cls_{\vp_1}\otimes H^2(\D^{n-1})}^{\cls}, P_{\clq_{\vp_1}\otimes {\clq_{\Phi}[2,n]}^{\perp}}^{\cls} \in \mathcal T(\cls)$ and $P_{\cls_{\vp_1}\otimes H^2(\D^{n-1})}^{\cls_{\Phi}}, P_{\clq_{\vp_1}\otimes {\clq_{\Phi}[2,n]}^{\perp}}^{\cls_{\Phi}} \in \mathcal T(\cls_{\Phi})$.
\end{lem}
\begin{proof}
First one observes that, by virtue of Lemma \ref{lemma-M1 and M2}, it is enough to prove the result for $\cls$. Note that $M_{\vp_1} \cls \subseteq \cls$. Define $R_{\vp_1} \in \clb(\cls)$ by $R_{\vp_1} = M_{\vp_1}|_{\cls}$. Then $R_{\vp_1} = \vp_1(M_{z_1})|_{\cls}$ and
\[
P_{M_{\vp_1} \cls} = R_{\vp_1} R_{\vp_1}^* \in \mathcal T(\cls).
\]
Now on the one hand
\[
\cls_{\vp_1}\otimes H^2(\D^{n-1}) = M_{\vp_1} H^2(\D^n) = M_{\vp_1} \cls \oplus \Big(M_{\vp_1} H^2(\D^n) \ominus M_{\vp_1} \cls \Big),
\]
on the other hand, $M_{\vp_1} H^2(\D^n) \ominus M_{\vp_1} \cls = M_{\vp_1} (H^2(\D^n) \ominus \cls)$ is finite dimensional, and hence we conclude $P_{\cls_{\vp_1}\otimes H^2(\D^{n-1})} \in \mathcal T(\cls)$. This along with $\mbox{dim~}(\cls\ominus \cls_{\Phi}) < \infty$ and the decomposition
\[
\cls=(\cls_{\vp_1}\otimes H^2(\D^{n-1})) \oplus (\clq_{\vp_1}\otimes {\clq_{\Phi}[2,n]}^{\perp}) \oplus (\cls\ominus \cls_{\Phi}),
\]
implies that $P_{\clq_{\vp_1}\otimes {\clq_{\Phi}[2,n]}^{\perp}} \in \mathcal T(\cls)$. This completes the proof of the lemma.
\end{proof}

For simplicity, let us introduce some more notation. Given $q \in \mathbb{N}$, let us denote
\[
H_q = \mathbb{C} \otimes \cdots \otimes \mathbb{C} \subseteq H^2(\D^q).
\]
Note that $H_q$ is the one-dimensional subspace consisting of the constant functions in $H^2(\D^q)$.
Recalling $\dim (\cls\ominus \cls_{\Phi}) = m (< \infty)$, we consider the orthogonal decomposition of $\cls_{\vp_1}\otimes H^2(\D^{n-1})$ as:
\[
\cls_{\vp_1}\otimes H^2(\D^{n-1})= \cls_1 \oplus \cls_2 \oplus \cls_3,
\]
where
\[
\begin{cases}
& \cls_1 = (\vp_1\clq_{z^m})\otimes H_{n-2} \otimes H^2(\D)
\\
& \cls_2 = \cls_{z^m\vp_1}\otimes H_{n-2} \otimes H^2(\D)
\\
& \cls_3 = \cls_{\vp_1}\otimes (H_{n-2} \otimes H^2(\D))^{\perp}.
\end{cases}
\]
Clearly $\cls_3 = \cls_{\vp_1} \otimes (\sum\limits_{i=1}^{n-2} z_i H^2(\D^{n-1}))$. Finally, we define
\[
\cll = \cls_2 \oplus \cls_3 \oplus \Big(\clq_{\vp_1}\otimes {\clq_{\Phi}[2,n]}^{\perp}\Big).
\]
With this notation we have
\[
\cls_{\Phi} = \cls_1 \oplus \cll,
\]
and
\[
\cls =(\cls \ominus \cls_{\Phi}) \oplus \cls_1 \oplus \cll.
\]

\begin{lem}
\label{projection 2}
$P_{\cls_i}^{\cls} \in \mathcal T(\cls)$ and $P_{\cls_i}^{\cls_{\Phi}} \in \mathcal T (\cls_{\Phi})$ for all $i = 1, 2, 3$.
\end{lem}
\begin{proof}
In view of Lemma \ref{lemma-M1 and M2}, it is enough to prove that $P_{\cls_i}^{\cls} \in \mathcal T(\cls)$, $i = 1, 2, 3$. Note that $P_{\cls_{\vp_1} \otimes H_{n-2} \otimes H^2(\D)} \in \clt(\cls)$ as
\[
P_{\cls_{\vp_1} \otimes H_{n-2} \otimes H^2(\D)} = P_{\cls_{\vp_1} \otimes H^2(\D^{n-1})} (I_{\cls} - X) P_{\cls_{\vp_1} \otimes H^2(\D^{n-1})},
\]
where 
\[
X = \sum \limits_{2 \leq i_1 < \cdots < i_k \leq n-1} (-1)^{k+1} R_{z_{i_1}} \cdots R_{z_{i_k}} R_{z_{i_1}}^* \cdots R_{z_{i_k}}^*. 
\]
Therefore
\[
P_{\cls_3} = P_{\cls_{\vp_1}\otimes H^2(\D^{n-1})} - P_{\cls_{\vp_1}\otimes H_{n-2} \otimes H^2(\D)} \in \clt(\cls).
\]
Finally, since $P_{\cls_2}=
R_{z_1}^m P_{\cls_{\vp_1} \otimes H_{n-2} \otimes H^2(\D)}R_{z_1}^{*m}$ and $\cls_1 \oplus \cls_2 = \cls_{\vp_1} \otimes H_{n-2} \otimes H^2(\D)$, it follows that $P_{\cls_1}$ and $P_{\cls_2}$ are in $\clt(\cls)$.
\end{proof}

Before we proceed to the unitary equivalence of the $C^*$-algebras $\mathcal T (\cls)$ and $\mathcal T (\cls_{\Phi})$ we note that
\[
\vp_1\clq_{z^m} = \mbox{span~} \{\vp_1, \vp_1 z, \ldots, \vp_1 z^{m-1}\}.
\]

\begin{thm}\label{main2}
$\mathcal T (\cls)$ and $\mathcal T (\cls_{\Phi})$ are unitarily equivalent.
 \end{thm}
\begin{proof}
By noting that $H^2(\D) = H_1 \oplus \cls_z$, we decompose $\cls_1$ as $\cls_1 = \clf_1 \oplus \clm_1$ where
\[
\clf_1 =(\vp_1\clq_{z^m}) \otimes H_{n-1}, \quad \mbox{and} \quad \clm_1 = (\vp_1\clq_{z^m}) \otimes H_{n-2} \otimes \cls_{z}.
\]
Taking into consideration $\mbox{dim} \clf_1 = \mbox{dim~} (\cls\ominus \cls_{\Phi})$, we have a unitary $V: \clf_1 \to \cls\ominus \cls_{\Phi}$, and then, using the decompositions
\[
\cls_{\Phi} = \clf_1 \oplus \clm_1 \oplus \cll.
\]
and
\[
\cls =(\cls \ominus \cls_{\Phi}) \oplus \cls_1 \oplus \cll,
\]
we see that
\[
U = \begin{bmatrix} V & 0 & 0 \\ 0 & M_{z_n}^* & 0 \\ 0 & 0 & I_{\cll} \end{bmatrix} : \clf_1 \oplus \clm_1 \oplus \cll \raro (\cls \ominus \cls_{\Phi}) \oplus \cls_1 \oplus \cll,
\]
defines a unitary from $\cls_{\Phi}$ to $\cls$. We claim that $U^*\mathcal T (\cls)U=\mathcal T (\cls_{\Phi})$. First we prove that $U^*\mathcal T (\cls)U \subseteq \mathcal T (\cls_{\Phi})$. Since $\mbox{dim} \clf_1  < \infty$, it suffices to prove that $U^*R_{z_i}^{\cls}U|_{\clm_1 \oplus \cll} \in \clt(\cls_{\Phi})$ for all $i=1,\cdots,n$. Observe first that $U \clm_1 = M_{z_n}^* \clm_1 = \cls_1 \subseteq \cls_{\Phi}$, $M_{z_n} \cls_1 \subseteq \cls_1$ and $M_{z_n} \cll \subseteq \cll$. Since
\[
U^*R_{z_n}^{\cls}U|_{\clm_1 \oplus \cll} =
U^* M_{z_n} M_{z_n}^*|_{\clm_1} +  M_{z_n}|_{\cll},
\]
and $U^* M_{z_n} M_{z_n}^*|_{\clm_1} =  M_{z_n}^2 M_{z_n}^*|_{\clm_1}
=  M_{z_n}^2 P_{\cls_{\Phi}} M_{z_n}^*|_{\clm_1}$, it follows that
\[
U^*R_{z_n}^{\cls}U|_{\clm_1 \oplus \cll}
= (R_{z_n}^{\cls_{\Phi}})^2 (R_{z_n}^{\cls_{\Phi}})^* P^{\cls_{\Phi}}_{\clm_1}
+ R_{z_n}^{\cls_{\Phi}} P^{\cls_{\Phi}}_{\cll} \in \clt(\cls_{\Phi}).
\]
Now for $1<i<n$, we have
\[
U^*R_{z_i}^{\cls}U|_{\clm_1 \oplus \cll} =
U^* M_{z_i} M_{z_n}^*|_{\clm_1} + U^* M_{z_i}|_{\cll},
\]
where $U^* M_{z_i} M_{z_n}^*|_{\clm_1} = M_{z_i} M_{z_n}^*|_{\clm_1}$ as
$z_i \cls_1 \subseteq \cls_3 \subseteq \cll$. On the other hand, since
$z_i \cls_2 \subseteq \cls_3$ we have $z_i \cll \subseteq \cll$
and hence $U^* M_{z_i}|_{\cll} = M_{z_i}|_{\cll}$, whence
\[
U^*R_{z_i}^{\cls}U|_{\clm_1 \oplus \cll} = R_{z_i}^{\cls_{\Phi}} (R_{z_n}^{\cls_{\Phi}})^* P^{\cls_{\Phi}}_{\clm_1} +
R_{z_i}^{\cls_{\Phi}} P^{\cls_{\Phi}}_{\cll} \in \clt(\cls_{\Phi}).
\]
Now we decompose $\clm_1$ as $\clm_1 = \clk_1 \oplus \tilde{\clk}_1$ where
\[
\clk_1 = (\vp_1\clq_{z^{m-1}})\otimes H_{n-2} \otimes \cls_{z} \quad \mbox{and} \quad \tilde{\clk}_1 = (\vp_1 {z^{m-1}} \mathbb{C})\otimes H_{n-2} \otimes \cls_{z}.
\]
Then
\[
U^*R_{z_1}^{\cls}U|_{\clm_1} = U^*M_{z_1} M_{z_n}^*|_{\clk_1}
+ U^*M_{z_1} M_{z_n}^*|_{\tilde{\clk}_1}
= M_{z_n} M_{z_1} M_{z_n}^*|_{\clk_1} +
M_{z_1} M_{z_n}^*|_{\tilde{\clk}_1},
\]
as $M_{z_1} M_{z_n}^* \clk_1 \subseteq \cls_1$ and $M_{z_1} M_{z_n}^* \tilde{\clk}_1 \subseteq \cls_2$. On the other hand,
$U^*R_{z_1}^{\cls}U|_{\cls_2 \oplus \cls_3} =
M_{z_1}|_{\cls_2 \oplus \cls_3}$ as
$M_{z_1} (\cls_2 \oplus \cls_3) \subseteq
\cls_2 \oplus \cls_3 \subseteq \cll$, and finally, by denoting
$\cln = \clq_{\vp_1}\otimes \clq_{\Phi}{[2,n]}^{\perp}$, we have
\[
U^*R_{z_1}^{\cls}U|_{\cln} = U^* M_{z_1}|_{\cln} =
U^*(I_{\cls} - P^{\cls}_{\cls_1}) M_{z_1}|_{\cln} +
U^* P^{\cls}_{\cls_1} M_{z_1}|_{\cln}.
\]
Then $\cls \ominus \cls_1 = (\cls \ominus \cls_{\Phi}) \oplus \cll$
and $M_{z_1}\cln\subseteq \cls_{\Phi}$ implies that
\[
U^*R_{z_1}^{\cls}U|_{\cln} = P^{\cls_{\Phi}}_{\cll} M_{z_1}|_{\cln}
+ M_{z_n} P^{\cls_{\Phi}}_{\cls_1} M_{z_1}|_{\cln} + F,
\]
and so
\[
\begin{split}
U^*R_{z_1}^{\cls}U|_{\clm_1 \oplus \cll}
& = R_{z_n}^{\cls_{\Phi}} R_{z_1}^{\cls_{\Phi}}
(R_{z_n}^{\cls_{\Phi}})^* P^{\cls_{\Phi}}_{\clk_1} +
R_{z_1}^{\cls_{\Phi}} (R_{z_n}^{\cls_{\Phi}})^* P^{\cls_{\Phi}}_{\tilde{\clk}_1}
+R_{z_1}^{\cls_{\Phi}} P^{\cls_{\Phi}}_{\cls_2 \oplus \cls_3}
\\
& \quad + P^{\cls_{\Phi}}_{\cll} R_{z_1}^{\cls_{\Phi}} P^{\cls_{\Phi}}_{\cln}
+ R_{z_n}^{\cls_{\Phi}} P^{\cls_{\Phi}}_{\cls_1} R_{z_1}^{\cls_{\Phi}}
P^{\cls_{\Phi}}_{\cln} + F.
\end{split}
\]
This implies that $U^*R_{z_1}^{\cls}U \in \clt(\cls_{\Phi})$, and therefore
$U^*\mathcal T (\cls)U \subseteq \mathcal T (\cls_{\Phi})$. We now proceed to prove
the reverse inclusion $U\clt(\cls_{\Phi})U^*\in \clt(\cls)$. Since $\mbox{dim} (\cls\ominus \cls_{\Phi}) < \infty$, it is enough to prove that
$UR_{z_i}^{\cls_{\Phi}}U^*|_{\cls_1 \oplus \cll} \in \clt(\cls)$
for all $i=1,\ldots,n$. Once again, note that
$U^*\cls_1=\clm_1\subseteq \cls_{\Phi}$, $z_n\clm_1\subseteq \clm_1$ and $z_n\cll\subseteq\cll$. Hence
\[
UR^{\cls_{\Phi}}_{z_n}U^*|_{\cls_1\oplus\cll} = UM_{z_n}^2|_{\cls_1}+ U M_{z_n}|_{\cll} = M_{z_n}|_{\cls_1}+ M_{z_n}|_{\cll},
\]
that is
\[
UR^{\cls_{\Phi}}_{z_n}U^*|_{\cls_1\oplus\cll} = R_{z_n}^{\cls}P^{\cls}_{\cls_1\oplus \cll}\in \clt(\cls).
\]
Now, for fixed $1<i<n$, we have $z_i\clm_1\subseteq \cls_3$ and $z_i\cll\subseteq\cll$. Then
\begin{align*}
U R_{z_i}^{\cls_{\Phi}}U^*|_{\cls_1\oplus\cll}
& = U M_{z_i}M_{z_n}|_{\cls_1} + UM_{z_i}|_{\cll}\\
&=M_{z_i}M_{z_n}|_{\cls_1} + M_{z_i}|_{\cll}\\
&= R^{\cls}_{z_i}R_{z_n}^{\cls}P^{\cls}_{\cls_1}+ R_{z_i}^{\cls}P_{\cll}
\in \clt(\cls).
\end{align*}
Finally, we consider the decomposition $\cls_1= \cls_1'\oplus \cls_1''$ where
\[
\cls_1' = (\vp_1\clq_{z^{m-1}})\otimes H_{n-2} \otimes H^2(\D)
\quad \mbox{and} \quad
\cls_1'' = (\vp_1 {z^{m-1}} \mathbb{C})\otimes H_{n-2} \otimes H^2(\D).
\]
Then
\[
\begin{split}
UR^{\cls_{\Phi}}_{z_1}U^*|_{\cls_1}& = U M_{z_1}M_{z_n}|_{\cls_1'}+
U M_{z_1}M_{z_n}|_{\cls_1''}
\\
& =M_{z_n}^*M_{z_1}M_{z_n}|_{\cls_1'} +M_{z_1}M_{z_n}|_{\cls_1''}
\\
& =M_{z_1}|_{\cls_1'}+M_{z_1}M_{z_n}|_{\cls_1''},
\end{split}
\]
as $z_1z_n\cls_1'\subseteq \clm_1$ and $z_1z_n\cls_1''\subseteq \cls_2$.
Moreover
\[
UR^{\cls_{\Phi}}_{z_1}U^*|_{\cls_2\oplus\cls_3}=
UM_{z_1}|_{\cls_2\oplus\cls_3}=M_{z_1}|_{\cls_2\oplus\cls_3},
\]
as $z_1(\cls_2\oplus\cls_3)\subseteq\cls_2\oplus\cls_3$. Denoting $\cln=\clq_{\vp_1}\otimes \clq_{\Phi}{[2,n]}^{\perp}$, as before,
it follows that
\[
UR^{\cls_{\Phi}}_{z_1}U^*|_{\cln}=
UP_{\clm_1}^{\cls_{\Phi}}M_{z_1}|_{\cln}+
U(I_{\cls_{\Phi}}-P_{\clm_1}^{\cls_{\Phi}})M_{z_1}|_{\cln},
\]
this in turn implies that
\[UR^{\cls_{\Phi}}_{z_1}U^*|_{\cln}=M_{z_n}^*P_{\clm_1}^{\cls}M_{z_1}|_{\cln}
+ P_{\cll}^{\cls}M_{z_1}|_{\cln}+ F,
\]
as $\cls_{\Phi}\ominus \clm_1=\clf_1\oplus\cll$ and $\clf_1$
is finite dimensional. Therefore
\[
\begin{split}
UR^{\cls_{\Phi}}_{z_1}U^*|_{\cls_1\oplus \cll} & =
R_{z_1}^{\cls}P^{\cls}_{\cls_1'}+ R^{\cls}_{z_1}R^{\cls}_{z_n}
P^{\cls}_{\cls_1''} + R_{z_1}^{\cls}P^{\cls}_{\cls_2\oplus\cls_3}
\\
&
\quad  +(R^{\cls}_{z_n})^*P_{\clm_1}^{\cls}M_{z_1}P^{\cls}_{\cln}
+ P_{\cll}^{\cls}R^{\cls}_{z_1}P^{\cls}_{\cln}+ F \in\clt(\cls).
\end{split}
\]
This completes the proof of the theorem.
\end{proof}

On combining Theorem~\ref{main1} and Theorem~\ref{main2}, we have the following:

\begin{thm}\label{thm-main cstar}
If $\cls$ is a finite co-dimensional invariant subspace of $H^2(\D^n)$, then $\mathcal T (\cls)$ and $\mathcal T (H^2(\D^n))$ are unitarily equivalent.
\end{thm}

In the case $n=2$, the proof of the above result is considerably simpler and direct than the one by Seto \cite{MS} (for instance, if  $n=2$, then $1 < i < n$ case does not appear in the proof of Theorem \ref{main2}).

\vspace{0.3in}

\noindent\textbf{Acknowledgement:} The research of the first named
author is supported by DST-INSPIRE Faculty Fellowship No.
DST/INSPIRE/04/2015/001094. The research of the third named author
is supported in part by NBHM (National Board of Higher Mathematics, India) grant NBHM/R.P.64/2014, and the Mathematical Research Impact Centric Support (MATRICS) grant, File No : MTR/2017/000522, by the Science and Engineering Research Board (SERB), Department of Science \& Technology (DST), Government of India.

\end{document}